\newif\ifLTEX
\LTEXtrue

\documentclass[a4paper]{amsart}

\usepackage{mathtools,amssymb}
\usepackage[abbrev]{amsrefs}
\usepackage{xparse}

\ifLTEX
\usepackage{graphicx}
\else
\usepackage[dvipdfmx]{graphicx}
\fi
 
\usepackage{hyperref}
\usepackage{subcaption}
\usepackage{enumitem}
\usepackage{amsthm}
\usepackage{pdfpages}
\usepackage[all]{xy} 
\usepackage{float}
\usepackage{marginfix}
\usepackage{bm}

\newtheorem{thm}{Theorem}[section]
\newtheorem{lem}[thm]{Lemma}

\newtheorem{cor}[thm]{Corollary}
\newtheorem{prop}[thm]{Proposition}

\theoremstyle{definition}

\newcommand{\R}{\mathbb{R}}

\newcommand{\Z}{\mathbb{Z}}

\newcommand{\M}{\mathrm{Mod}_{0,2n+2}}

\newcommand{\LM}{\mathrm{LMod}_{2n+2}}

\newcommand{\Mg}{\mathrm{Mod}_{g}}

\newcommand{\Mgp}{\mathrm{Mod}_{g,n}}
\newcommand{\SM}{\mathrm{SMod}_{g;k}}

\newcommand{\Hg}{\mathcal{H}_{g}}
\newcommand{\SH}{\mathcal{SH}_{g;k}}

\newcommand{\Hil}{\bm{\mathrm{H}}_{2n+2}}
\newcommand{\LH}{\bm{\mathrm{LH}}_{2n+2}}
\newcommand{\LHk}{\bm{\mathrm{LH}}_{2n+2;k}}
\newcommand{\LHt}{\bm{\mathrm{LH}}_{2n+2;2}}

\newcommand{\A}{\mathcal{A}}
\newcommand{\B}{\mathcal{B}}

\numberwithin{equation}{section}

\allowdisplaybreaks
\title[Small generating set for balanced superelliptic handlebody group]{A small generating set for the balanced superelliptic handlebody group}

\author[G.~Omori]{Genki Omori}
\address{
(Genki Omori)
Department of Mathematics, Faculty of Science and Technology, Tokyo University of Science, 2641 Yamazaki, Noda-shi, Chiba, 278-8510 Japan
}
\email{omori\_genki@ma.noda.tus.ac.jp}


\date{\today}

\begin{document}
\maketitle
\begin{abstract}
The balanced superelliptic handlebody group is the normalizer of the transformation group of the balanced superelliptic covering space in the handlebody group of the total space.  
We prove that the balanced superelliptic mapping class group is generated by four elements. 
To prove this, we also proved that the liftable Hilden group is generated by three elements. 
This generating set for the liftable Hilden group is minimal except for some hyperelliptic cases and the generating set for the balanced superelliptic mapping class group above is also minimal for several cases.  
\end{abstract}

\section{Introduction}

Let $H_{g}$ be an oriented 3-dimensional handlebody of genus $g\geq 0$ and $B^3=H_0$ a 3-ball. 
We write $\Sigma _{g}=\partial H_g$ and $S^2=\partial B^3$, respectively. 
For a subset $A$ of $\Sigma _g$ $(g\geq 0)$, the {\it mapping class group} $\mathrm{Mod}(\Sigma _{g}, A)$ of the pair $(\Sigma _{g}, A)$ is the group of isotopy classes of orientation-preserving self-homeomorphisms on $\Sigma _{g}$ which preserve $A$ setwise. 
When $A$ is a set of distinct $n$ points, we denote the mapping class group by $\mathrm{Mod}_{g,n}$. 
We denote $\mathrm{Mod}_{g,0}$ simply by $\mathrm{Mod}_{g}$. 

Humphries~\cite{Humphries} proved that $\Mgp $ is generated by $2g+1$ Dehn twists for $g\geq 2$ and $n\in \{ 0,1\}$ and the generating set is minimal in generating sets for $\Mgp$ which consist of Dehn twists. 
We focus on the study of minimal generating sets for mapping class groups.
Wajnryb~\cite{Wajnryb} proved that $\mathrm{Mod}_{g,n}$ is generated by two elements for $g\geq 1$ and $n\in \{0,1\}$. 
Since $\mathrm{Mod}_{g,n}$ is not a cyclic group, this Wajnryb's generating set is minimal. 
After that, Korkmaz~\cite{Korkmaz} proved that $\mathrm{Mod}_{g,n}$ is generated by two elements whose one element is a Dehn twist for $g\geq 1$ and $n\in \{0,1\}$. 
Monden~\cite{Monden} showed that $\mathrm{Mod}_{g,n}$ is generated by two elements for $g\geq 3$ and $n\geq 0$, and this generating set is also minimal. 

The \textit{handlebody group} $\mathcal{H}_{g}$ is the group of isotopy classes of orientation-preserving self-homeomorphisms on $H_{g}$. 
Suzuki~\cite{Suzuki} gave a generating set for $\Hg $ which consists of six elements for $g\geq 3$ (actually, the number of the generators can be reduced to five). 
We have a well-defined injective homomorphism $\mathcal{H}_{g}\hookrightarrow \Mg $ by restricting the actions of elements in $\mathcal{H}_{g}$ to $\Sigma _g$ and using the irreducibility of $H_g$.  
By this injective homomorphism, we regard $\Hg$ as the subgroup of $\Mg $ whose elements extend to $H_g$.  

For integers $n\geq 1$ and $k\geq 2$ with $g=n(k-1)$, the \textit{balanced superelliptic covering map} $p=p_{g,k}\colon H_g\to B^3$ is a $k$-fold branched covering map with the covering transformation group generated by the \textit{balanced superelliptic rotation} $\zeta =\zeta _{g,k}$ of order $k$ (precisely defined in Section~\ref{section_bscov} and see Figure~\ref{fig_bs_periodic_map}). 
The branch points set $\A \subset B^3$ of $p$ is the disjoint union of $n+1$ proper arcs in $B^3$ and the restriction of $p$ to the preimage $\widetilde{\A }$ of $\A $ is injective (i.e. $\widetilde{\A }$ is the fixed point set of $\zeta $). 
The restriction $p|_{\Sigma _g}\colon \Sigma _g\to S^2=\Sigma _0$ is also a $k$-fold branched covering with the branch points set $\B =\partial \A $ and we also call $p|_{\Sigma _g}$ the balanced superelliptic covering map. 
When $k=2$, $\zeta |_{\Sigma _g}$ coincides with a hyperelliptic involution, and for $k\geq 3$, the balanced superelliptic covering space was introduced by Ghaswala and Winarski~\cite{Ghaswala-Winarski2}. 
We often abuse notation and simply write $p|_{\Sigma _g}=p$ and $\zeta |_{\Sigma _g}=\zeta $, respectively.  

For $g=n(k-1)\geq 1$, an orientation-preserving self-homeomorphism $\varphi $ on $\Sigma _{g}$ or $H_g$ is \textit{symmetric} for $\zeta =\zeta _{g,k}$ if $\varphi \left< \zeta \right> \varphi ^{-1}=\left< \zeta \right> $. 
The \textit{balanced superelliptic mapping class group} (or the \textit{symmetric mapping class group}) $\SM $ is the subgroup of $\Mg $ which consists of elements represented by symmetric homeomorphisms. 
In particular, $\mathrm{SMod}_{g;2}$ is called the \emph{hyperelliptic mapping class group}. 
Birman and Hilden~\cite{Birman-Hilden3} showed that $\SM $ coincides with the group of symmetric isotopy classes of symmetric homeomorphisms on $\Sigma _g$. 
We call the intersection $\SH =\SM \cap \Hg $ the \textit{balanced superelliptic handlebody group} (or the \textit{symmetric handlebody group}). 
By Lemma~1.21 in \cite{Iguchi-Hirose-Kin-Koda} and Lemma~2.1 in \cite{Omori-Yoshida}, $\SH $ is also isomorphic to the group of symmetric isotopy classes of symmetric homeomorphisms on $H_g$. 
When $k=2$, Stukow~\cite{Stukow} gave a minimal generating set for $\mathrm{SMod}_{g;2}$ which consists of two torsion elements. 
For the case $k\geq 3$, the author~\cite{Omori} proved that $\SM $ is generated by three elements and this generating set is minimal except for the case of even $n$.

We regard $\mathrm{Mod}(S^2,\B )$ as $\M $ 
and let $\Hil $ be the group of isotopy classes of orientation-preserving self-homeomorphisms on $B^3$ fixing $\A $ setwise. 
It is a well-known result that $\M $ is generated  by two elements and this generating set is minimal. 
The group $\Hil $ is introduced by Hilden~\cite{Hilden} and is called the \textit{Hilden group}. 
He gave a finite generating set for $\Hil $ in~\cite{Hilden}. 
Tawn~\cite{Tawn1} gave a finite presentation for $\Hil $ whose generating set is smaller than Hilden's generating set. 
By restricting the actions of elements in $\Hil $ to $S^2$, we have an injective homomorphism $\Hil \hookrightarrow \M $ (see \cite[p.~157]{Brendle-Hatcher} or \cite[p.~484]{Hilden}) and regard $\Hil $ as the subgroup of $\M $ whose elements extend to homeomorpisms on $B^3$ which preserve $\A $ by this injective homomorphism. 
Since elements in $\SM $ (resp. in $\SH $) are represented by elements which preserve $\B $ (resp. $\A $) by the definitions, we have homomorphisms $\theta \colon \SM \to \M $ and $\theta |_{\SH } \colon \SH \to \Hil $ that are introduced by Birman and Hilden~\cite{Birman-Hilden2}. 
They also proved that $\theta (\mathrm{SMod}_{g;2} )=\M $, and Hirose and Kin~\cite{Hirose-Kin} showed that $\theta (\mathcal{SH}_{g;2})=\Hil $. 

A self-homeomorphism $\varphi $ on $\Sigma _{0}$ (resp. on $B^3$) is \textit{liftable} with respect to $p=p_{g,k}$ if there exists a self-homeomorphism $\widetilde{\varphi }$ on $\Sigma _{g}$ (resp. on $H_g$) such that $p\circ \widetilde{\varphi }=\varphi \circ p$, namely, the following diagrams commute: 
\[
\xymatrix{
\Sigma _g \ar[r]^{\widetilde{\varphi }} \ar[d]_p &  \Sigma _{g} \ar[d]^p & H_g \ar[r]^{\widetilde{\varphi }}\ar[d]_{p}  &  H_{g}\ar[d]^{p} \\
\Sigma _{0}  \ar[r]_{\varphi } &\Sigma _{0}, \ar@{}[lu]|{\circlearrowright} & B^3  \ar[r]_{\varphi } &B^3. \ar@{}[lu]|{\circlearrowright}
}
\] 
The \textit{liftable mapping class group} $\mathrm{LMod}_{2n+2;k}$ is the subgroup of $\M $ which consists of elements represented by liftable homeomorphisms on $S^2$ for $p|_{S^2}$, and the \textit{liftable Hilden group} $\bm{\mathrm{LH}}_{2n+2;k}$ is the subgroup of $\Hil $ which consists of elements represented by liftable homeomorphisms on $B^3$ for $p$. 
As a homomorphic image in $\M $, we have $\bm{\mathrm{LH}}_{2n+2;k}=\mathrm{LMod}_{2n+2;k}\cap \Hil $ by Lemma~2.2 in \cite{Omori-Yoshida}. 
By the definitions, we have homomorphisms $\theta \colon \SM \to \mathrm{LMod}_{2n+2;k}$ and $\theta |_{\SH }\colon \SH \to \bm{\mathrm{LH}}_{2n+2;k}$. 
Birman and Hilden~\cite{Birman-Hilden2} proved that $\theta $ induces an isomorphism $\mathrm{LMod}_{2n+2;k}\cong \SM /\left< \zeta \right> $, and Hirose and Kin~\cite{Hirose-Kin} showed that $\theta |_{\mathcal{SH}_{g;2}}$ induces an isomorphism $\bm{\mathrm{LH}}_{2n+2;2}=\Hil \cong \mathcal{SH}_{g;2}/\left< \zeta _{g,2}\right> $ (remark that $\mathrm{LMod}_{2n+2;2}=\M $). 
For the case of the handlebody subgroup and $k\geq 3$, the author and Yoshida~\cite[Lemma~2.3]{Omori-Yoshida} showed that $\theta |_{\SH }$ induces an isomorphism $\bm{\mathrm{LH}}_{2n+2;k}\cong \SH /\left< \zeta \right> $.  
For $k\geq 3$, the author~\cite{Omori} proved that $\mathrm{LMod}_{2n+2;k}$ is generated by three elements and this generating set is minimal except for the case of even $n$.

The main theorem in this paper is as follows. 

\begin{thm}\label{thm_lmod}
For $k\geq 2$, $\bm{\mathrm{LH}}_{2n+2;k}$ is generated by three elements. 
\end{thm}

Theorem~\ref{thm_lmod} is proved in Sections~\ref{section_lmod}. 
We remark that $\LHt =\Hil $ and $\LHk =\bm{\mathrm{LH}}_{2n+2;l}$ for $k, l\geq 3$ by Lemma~\ref{lem_GW} (that is Lemma~3.6 in \cite{Ghaswala-Winarski1}). 
Hence we omit ``$k$'' in the notation of the liftable Hilden group for $k\geq 3$ (i.e. we express $\bm{\mathrm{LH}}_{2n+2;k}$ as $\LH $ for $k\geq 3$). 
The next corollary follows immediately from Theorem~\ref{thm_lmod} and the following exact sequence which is obtained from Theorem~2.11 in~\cite{Hirose-Kin} and Lemma~2.3 in~\cite{Omori-Yoshida}:  
\begin{eqnarray*}\label{exact_SH_handlebody}
1\longrightarrow \left< \zeta \right> \longrightarrow \SH \stackrel{\theta }{\longrightarrow }\LHk \longrightarrow 1. 
\end{eqnarray*}

\begin{cor}\label{thm_smod}
Assume that $g=n(k-1)$ for $n\geq 1$ and $k\geq 2$. 
Then, $\SH $ is generated by four elements. 
\end{cor}

For explicit generators for $\LHk $ in Theorem~\ref{thm_lmod} and their lifts to $H_g$ in Corollary~\ref{thm_smod}, see Proposition~\ref{prop_LM}, Appendix in \cite{Hirose-Kin}, and Section~5.1 in~\cite{Omori-Yoshida}. 

The integral first homology group $H_1(G)$ of a group $G$ is isomorphic to the abelianization of $G$. 
By Corollary~A.9 in \cite{Hirose-Kin}, Theorems~1.1 and 1.2 in~\cite{Omori-Yoshida}, and a computation of aberianization of $\Hil $ from the presentation in~\cite{Tawn1}, the integral first homology groups of $\LHk $ and $\SH $ for $k\geq 2$ are as follows (remark that Tawn gave a presentation for the Hilden group of one marked disk case in~\cite{Tawn1}. The group $\Hil $ is a quotient of its Hilden group):

\begin{thm}[\cite{Tawn1} for the case of $k=2$ and \cite{Omori-Yoshida} for the case of $k\geq 3$]\label{thm_abel_lmod}
For $n\geq 1$ and $k\geq 2$, 
\[
H_1(\LHk )\cong \left\{ \begin{array}{ll}
 \Z \oplus \Z _2&\text{if }  k=2\text{ and }n \text{ is odd},   \\
 \Z \oplus \Z _{2}\oplus \Z _{2}& \text{otherwise}.
 \end{array} \right.
\]
\end{thm}

\begin{thm}[\cite{Hirose-Kin} for the case of $k=2$ and \cite{Omori-Yoshida} for the case of $k\geq 3$]\label{thm_abel_smod}
For $n\geq 1$ and $k\geq 2$ with $g=n(k-1)$,
\[
H_1(\SH )\cong \left\{ \begin{array}{ll}
 \Z \oplus \Z _2\oplus \Z _{2}\oplus \Z _{2}&\text{if }  k\geq 4 \text{ is even and }n \text{ is odd},   \\
 \Z \oplus \Z _{2}\oplus \Z _{2}& \text{otherwise}.
 \end{array} \right.
\]
\end{thm}

For a group $G$, the minimal number of generators for $H_1(G)$ gives a lower bound of the minimal number of generators for $G$.  
By Theorem~\ref{thm_abel_lmod}, we see that the generating set for $\LHk $ in Theorem~\ref{thm_lmod} is minimal except for $k=2$ and odd $n$. 
Similarly, the generating set for $\SH $ in Corollary~\ref{thm_smod} is minimal for even $k\geq 4$ and odd $n$ by Theorem~\ref{thm_abel_smod}. 

\section{Preliminaries}\label{Preliminaries}

\subsection{The balanced superelliptic covering space}\label{section_bscov}

In this section, we review the definition of the balanced superelliptic covering space from Section~2.1 in \cite{Omori-Yoshida}. 
For integers $n\geq 1$, $k\geq 2$, and $g=n(k-1)$, we describe the handlebody $H_{g}$ as follows. 
We take the unit 3-ball $B(1)$ in $\R ^3$ and $n$ mutually disjoint parallel copies $B(2),\ B(3),\ \dots ,\ B(n+1)$ of $B(1)$ by translations along the x-axis such that 
\[
\max \bigl( B(i)\cap (\R \times \{ 0\}\times \{ 0\} )\bigr) <\min \bigl( B(i+1)\cap (\R \times \{ 0\}\times \{ 0\} )\bigr)
\]
in $\R =\R \times \{ 0\}\times \{ 0\} $ for $1\leq i\leq n$ (see Figure~\ref{fig_bs_periodic_map}). 
Let $\zeta$ be the rotation of $\R ^3$ by $-\frac{2\pi }{k}$ about the $x$-axis. 
Then for each $1\leq i\leq n$, we connect $B(i)$ and $B(i+1)$ by $k$ 3-dimensional 1-handles such that the union of the $k$ 3-dimensional 1-handles are preserved by the action of $\zeta $ as in Figure~\ref{fig_bs_periodic_map}. 
Since the union of $B(1)\cup B(2)\cup \cdots \cup B(n+1)$ and the attached $n\times k$ 3-dimensional 1-handles is homeomorphic to $H_g(=H_{n(k-1)})$, we regard this union as $H_g$. 

\begin{figure}[h]
\includegraphics[scale=1.35]{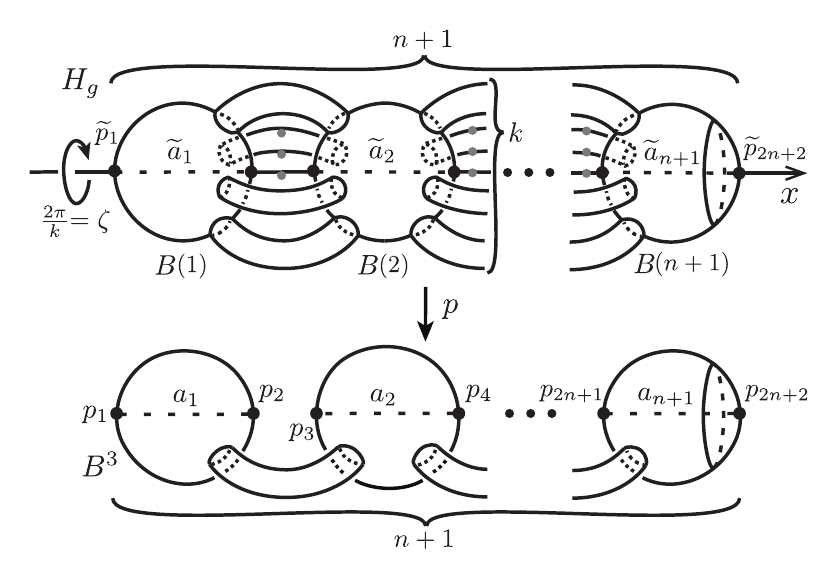}
\caption{The balanced superelliptic covering map $p=p_{g,k}\colon H_g\to B^3$.}\label{fig_bs_periodic_map}
\end{figure}

By the construction above, the action of $\zeta $ on $\R ^3$ induces the action on $H_g$ and the fixed points set of $\zeta =\zeta |_{H_g}$ is $\widetilde{\A }=H_g\cap (\R \times \{ 0\}\times \{ 0\})$. 
We call $\zeta $ the \textit{balanced superelliptic rotation} on $H_g$. 
We can see that the intersection $\widetilde{a}_i=B(i)\cap \widetilde{\A }$ for $1\leq i\leq n+1$ is a proper simple arc in $H_g$ and $\widetilde{\A }=\widetilde{a}_1\sqcup \widetilde{a}_2\sqcup \cdots \sqcup \widetilde{a}_{n+1}$ (see Figure~\ref{fig_bs_periodic_map}). 
The quotient space $H_g/\left< \zeta \right>$ is homeomorphic to $B^3$ and the induced quotient map $p=p_{g,k}\colon H_g\to B^3$ is a branched covering map with the branch points set $\A =p(\widetilde{\A })=p(\widetilde{a}_1)\sqcup p(\widetilde{a}_2)\sqcup \cdots \sqcup p(\widetilde{a}_{n+1})\subset B^3$. 
We call $p$ 
 the \textit{balanced superelliptic covering map}. 
Put 
\begin{itemize}
\item $a_i=p(\widetilde{a}_i)$ \quad for $1\leq i\leq n+1$, 
\item $\widetilde{p}_{2i-1}=\min \widetilde{a}_i$ and $\widetilde{p}_{2i}=\max \widetilde{a}_i$ \quad in $\R =\R \times \{ 0\}\times \{ 0\}$ for $1\leq i\leq n+1$, 
\item $p_i=p(\widetilde{p}_{i})$ \quad for $1\leq i\leq 2n+2$, 
\item $\B =\partial \A =\{ p_1, p_2, \dots , p_{2n+2}\}$, 
\item $\Sigma _g=\partial H_g$, \quad and \quad $S^2=\partial B^3$ (see Figure~\ref{fig_bs_periodic_map}).
\end{itemize}  
Then we also call the restriction $p|_{\Sigma _g}\colon \Sigma _g\to S^2$ the balanced superelliptic covering map and we often simply write $p|_{\Sigma _g}=p$. 
We note that the branch points set of $p\colon \Sigma _g\to S^2$ coincides with $\B $.

\subsection{Generators for the Hilden group and the spherical wicket group}\label{section_wicket-group}

In this section, we review generators for $\Hil =\LHt $ via homomorphic images from the spherical wicket group. 

Let $SB_{2n+2}$ be the \textit{spherical braid group} of $2n+2$ strands. 
We regard an element in $SB_{2n+2}$ as a $(2n+2)$-tangle in $S^2\times [0,1]$ which consists of $2n+2$ simple proper arcs whose one of the endpoints lies in $\B \times \{ 0\}$ and the other one lies in $\B \times \{ 1\}$, and we also regard $\A $ as a $(n+1)$-tangle in $B^3 $. 
Such $\A $ is called a \textit{wicket}. 
For $b\in SB_{2n+2}$, denote by $^b\! \A$ the $(n+1)$-tangle in $B^3 $ which is obtained from $b$ by attaching a copy of $B^3 $ with $\A $ to $S^2\times \{ 0\}$ such that $p_i\in \partial B^3$ is attached to the end of $i$-th strand in $b$ (see Figure~\ref{fig_wicket}), where we regard the manifold obtained by attaching the copy of $B^3$ to $S^2\times [0,1]$ along $S^2\times \{ 0\} $ as $B^3$. 
The \textit{spherical wicket group} $SW_{2n+2}$ is the subgroup of $SB_{2n+2}$ whose element $b$ satisfies the condition that $^b\! \A$ is isotopic to $\A $ relative to $\partial B^3=S^2\times \{ 1\}$. 
Brendle and Hatcher~\cite{Brendle-Hatcher} introduced the group $SW_{2n+2}$ and gave its finite presentation. 

\begin{figure}[h]
\includegraphics[scale=0.5]{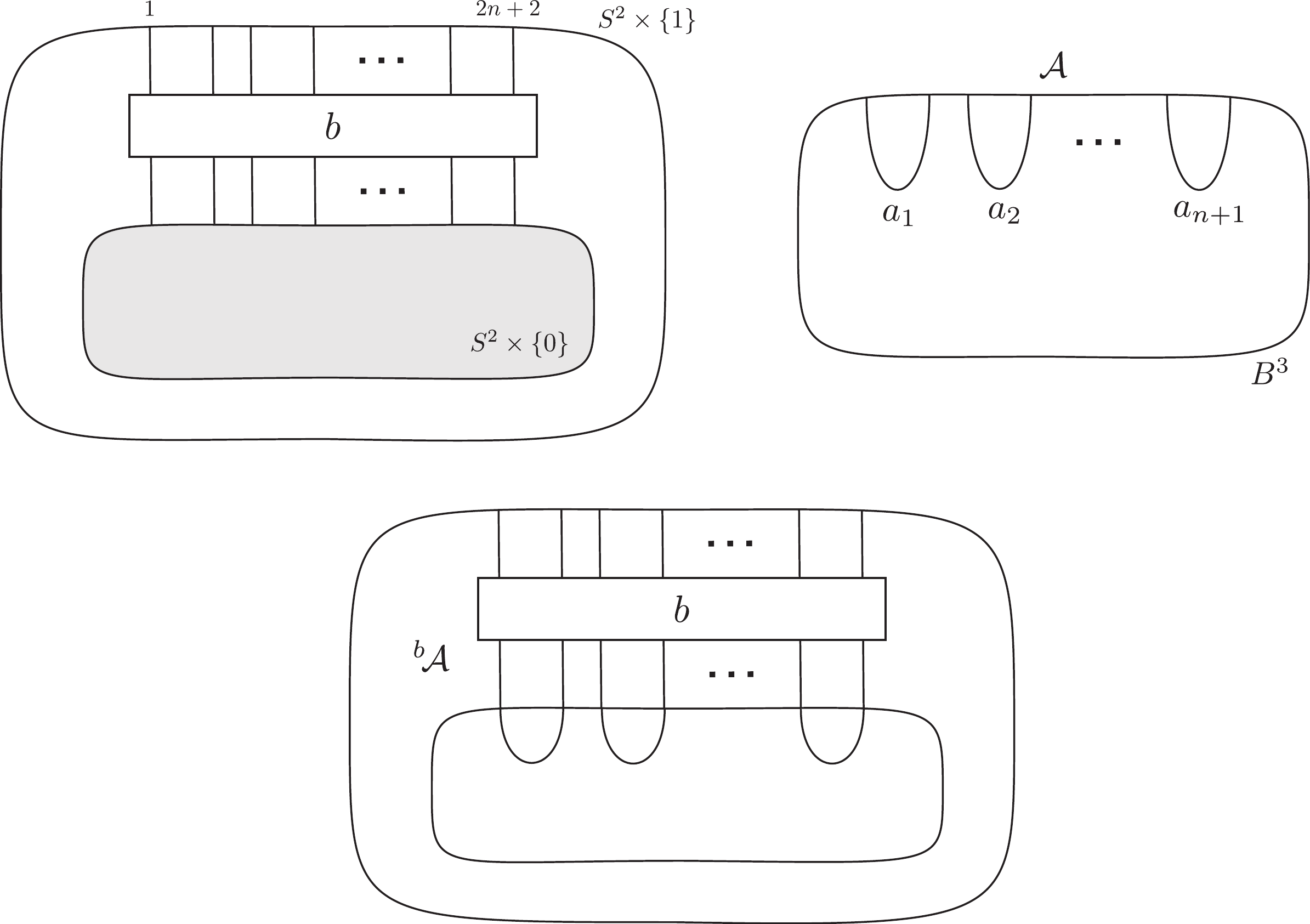}
\caption{A spherical braid $b$ and the tangles $\A $ and $^b\! \A $ in $B^3$.}\label{fig_wicket}
\end{figure}

Let $\sigma _i\in SB_{2n+2}$ be the half-twist about $i$-th and $(i+1)$-st strands as on the left-hand side in Figure~\ref{fig_braid_product_def}.  
It is well-known that $SB_{2n+2}$ is generated by $\sigma _1,\ \sigma _2,\ \dots ,\ \sigma _{2n+1}$. 
We can check that $\sigma _{2i-1}\in SW_{2n+2}$ for $1\leq i\leq n+1$. 
For $b_1$, $b_2\in SB_{2n+2}$, the product $b_1b_2$ is a braid as on the right-hand side in Figure~\ref{fig_braid_product_def}. 

\begin{figure}[h]
\includegraphics[scale=0.75]{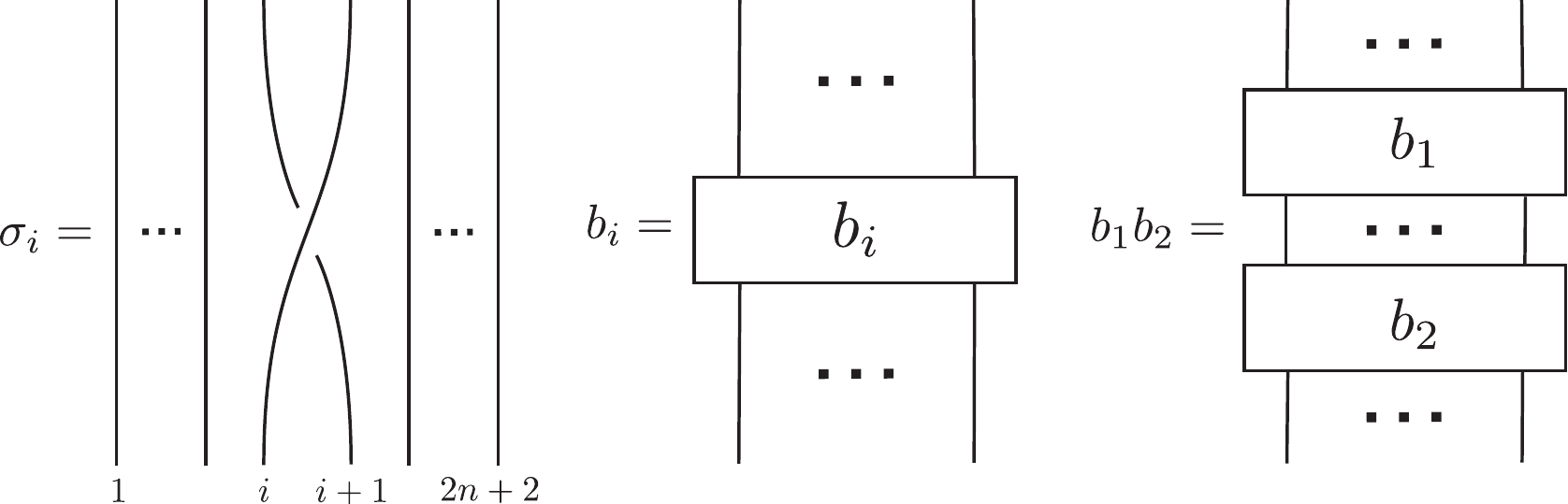}
\caption{The half-twist $\sigma_i\in SB_{2n+2}$ $(1\leq i\leq 2n+1)$ and the product $b_1b_2$ in $SB_{2n+2}$ for $b_i$ $(i=1,\ 2)$.}\label{fig_braid_product_def}
\end{figure}

Let $s_i$ and $r_i$ for $1\leq i\leq n$ be the $2n+2$ strands braids as in Figure~\ref{fig_r_i-s_i}. 
We see that $s_i$ and $r_i$ lie in $SW_{2n+2}$. 
Remark that the relations 
\begin{eqnarray*}
s_i=\sigma _{2i}\sigma _{2i+1}\sigma _{2i-1}\sigma _{2i} \quad  \text{and} \quad r_i=\sigma _{2i}^{-1}\sigma _{2i+1}^{-1}\sigma _{2i-1}\sigma _{2i}
\end{eqnarray*}
for $1\leq i\leq n$ hold in $SB_{2n+2}$. 
Brendle and Hatcher~\cite{Brendle-Hatcher} gave the following proposition. 

\begin{prop}\label{prop_Brendle-Hatcher}
For $n\geq 1$, $SW_{2n+2}$ is generated by $s_i$, $r_i$ for $1\leq i\leq n$, and $\sigma _{2j-1}$ for $1\leq j\leq n+1$. 
\end{prop}

\begin{figure}[h]
\includegraphics[scale=0.75]{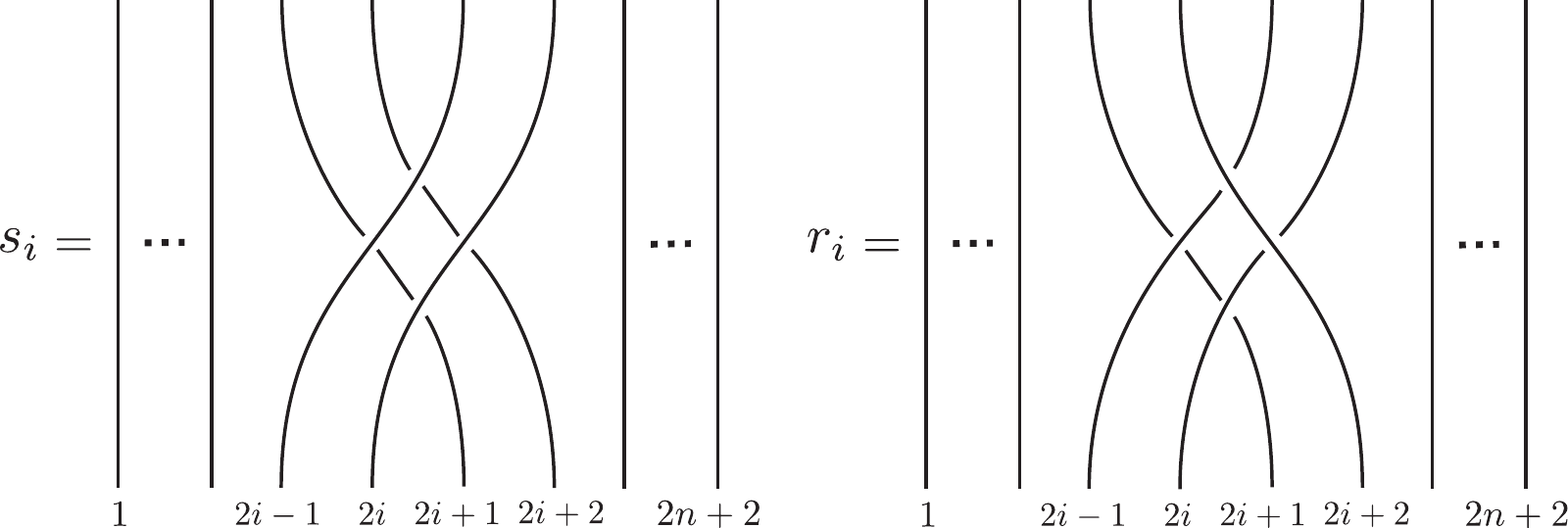}
\caption{The braids $s_i$ and $r_i$ in $SB_{2n+2}$ for $1\leq i\leq n$.}\label{fig_r_i-s_i}
\end{figure}

Let $l_i$ $(1\leq i\leq 2n+1)$ be an oriented simple arc on $S^2$ whose endpoints are $p_i$ and $p_{i+1}$ as in Figure~\ref{fig_path_l}. 
Put $L=l_1\cup l_2\cup \cdots \cup l_{2n+1}$. 
The isotopy class of a homeomorphism $\varphi $ on $\Sigma _0$ (resp. $B^3$) relative to $\B $ (resp. $\A $) is determined by the isotopy class of the image of $L$ by $\varphi $ relative to  $\B $. 
We identify $B^3$ with the 3-manifold with a sphere boundary on the lower side in Figure~\ref{fig_path_l} by some homeomorphism. 

\begin{figure}[h]
\includegraphics[scale=1.5]{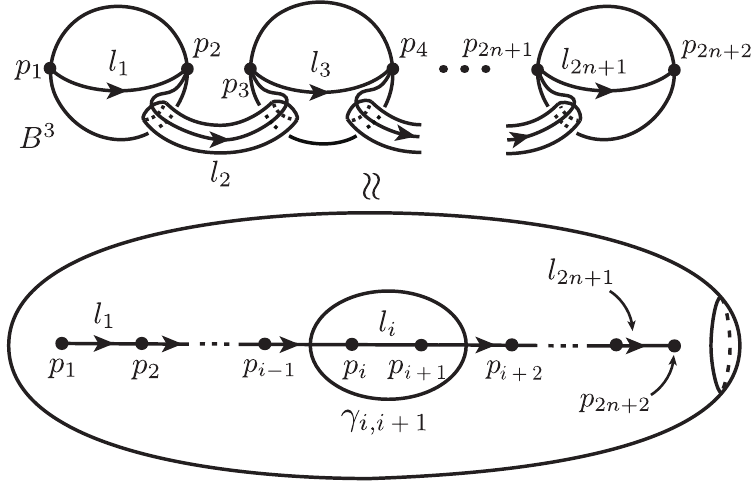}
\caption{A natural homeomorphism of $B^3$ and arcs $l_1,$ $l_2,\ \dots ,$ $l_{2n+1}$ and a simple closed curve $\gamma _{i,i+1}$ on $S^2$ for $1\leq i\leq 2n+1$.}\label{fig_path_l}
\end{figure}

Let $\sigma [l_i]$ for $1\leq i\leq 2n+1$ be a self-homeomorphism on $S^2$ which is described as the result of anticlockwise rotation of $l_i$ by $\pi $ in the regular neighborhood of $l_i$ in $S^2$ as in Figure~\ref{fig_sigma_l}. 
The self-homeomorphism $\sigma [l_i]$ is called the \textit{half-twist} along $l_i$. 
It is well-known that $\M $ is generated by $\sigma [l_1]$, $\sigma [l_2],\ \dots $, $\sigma [l_{2n+1}]$ (see for instance Section~9.1.4 in \cite{Farb-Margalit}). 
For maps or mapping classes $f$ and $g$, the product $gf$ means that $f$ is applied first. 
Then we have the surjective homomorphism 
\[
\Gamma \colon SB_{2n+2}\to \M 
\]
which is defined by $\Gamma (\sigma _i)=\sigma [l_i]$ for $1\leq i\leq 2n+1$.
The homomorphism $\Gamma $ has a kernel with order 2 which is generated by the full twist braid (see for instance Section~9.1.4 in \cite{Farb-Margalit}). 
We abuse notation and simply denote $\Gamma (b)$ for $b\in SB_{2n+2}$ by $b$ (i.e. we write $\sigma [l_i]=\sigma _i$, $\Gamma (s_i)=s_i$, and $\Gamma (r_i)=r_i$ in $\M $). 
Since $\Gamma (SW_{2n+2})=\Hil $ by Theorem~2.6 in \cite{Hirose-Kin}, we have the following proposition. 

\begin{prop}\label{prop_gen_Hilden-grp}
For $n\geq 1$, $\Hil $ is generated by $s_i$, $r_i$ for $1\leq i\leq n$, and $\sigma _{2j-1}$ for $1\leq j\leq n+1$. 
\end{prop}

\begin{figure}[h]
\includegraphics[scale=1.1]{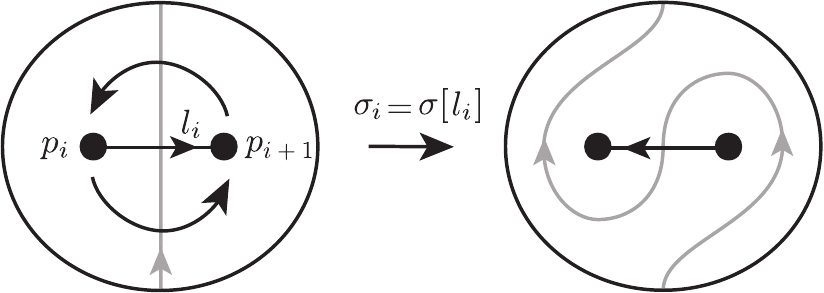}
\caption{The half-twist $\sigma [l_i]=\sigma _i$ for $1\leq i\leq 2n+1$.}\label{fig_sigma_l}
\end{figure}

\subsection{Generators for the liftable Hilden group}\label{section_liftable-element}

Assume that $k\geq 3$ in Section~\ref{section_liftable-element}. 
In this section, we review the generating set for the liftable Hilden group $\LHk =\LH $ for $k\geq 3$ of the presentation in Theorem~4.1 of~\cite{Omori-Yoshida} as homomorphic images from $SW_{2n+2}$. 
First, we review Ghaswala-Winarski's necessary and sufficient condition for lifting a homeomorphism on $S^2$ with respect to $p=p_{g,k}$ for $k\geq 3$. 
Since $\M $ is naturally acts on $\B =\{ p_1,\ p_2,\ \dots ,\ p_{2n+2}\}$, we have a surjective homomorphism 
\[
\Psi \colon \M \to S_{2n+2}
\]
given by $\Psi (\sigma _i)=(i\ i+1)$, where $S_{2n+2}$ is the symmetric group of degree $2n+2$. 

Put $\B _o=\{ p_1,\ p_3,\ \dots ,\ p_{2n+1}\}$ and $\B _e=\{ p_2,\ p_4,\ \dots ,\ p_{2n+2}\}$. 
An element $\sigma $ in $S_{2n+2}$ is \textit{parity-preserving} if $\sigma (\B _o)=\B _o$, and is \textit{parity-reversing} if $\sigma (\B _o)=\B _e$. 
An element $f$ in $\M $ is \textit{parity-preserving} (resp. \textit{parity-reversing}) if $\Psi (f)$ is \textit{parity-preserving} (resp. \textit{parity-reversing}). 
Let $W_{2n+2}$ be the subgroup of $S_{2n+2}$ which consists of parity-preserving or parity-reversing elements. 
Ghaswala and Winarski~\cite{Ghaswala-Winarski1} proved the following lemma. 
\begin{lem}[Lemma~3.6 in \cite{Ghaswala-Winarski1}]\label{lem_GW}
Let $\LM $ be the liftable mapping class group for the balanced superelliptic covering map $p_{g,k}$ for $n\geq 1$ and $k\geq3$ with $g=n(k-1)$. 
Then we have
\[
\LM =\Psi ^{-1}(W_{2n+2}).
\]
\end{lem}
Lemma~\ref{lem_GW} implies that a mapping class $f\in \M $ lifts with respect to $p=p_{g,k}$ if and only if $f$ is parity-preserving or parity-reversing. 
By Lemma~2.2 in~\cite{Omori-Yoshida}, we have $\LH =\Hil \cap \LM $, namely the liftability for $p$ is equivalent to one for $p|_{\Sigma _g}$. 
Moreover, by Lemma~\ref{lem_GW}, when $k\geq 3$, the liftability of a homeomorphism on $S^2$ does not depend on $k$. 
Hence we omit ``$k$'' in the notation of the liftable mapping class group and the liftable Hilden group for $k\geq 3$ (i.e. we express $\mathrm{LMod}_{2n+2;k}=\LM $ and $\bm{\mathrm{LH}}_{2n+2;k}=\LH $ for $k\geq 3$). 

Recall that $s_i$, $r_i$ for $1\leq i\leq n$, and $\sigma _{2j-1}$ for $1\leq j\leq n+1$ generates $\Hil $ by Proposition~\ref{prop_gen_Hilden-grp}. 
We denote $t_j=\sigma _{2j-1}^2$ for $1\leq j\leq n+1$ and $r=\sigma _{1}\sigma _{3}\cdots \sigma _{2n+1}$. 
Then we can see that $\Psi (s_i)=\Psi (r_i)=(2i-1\ 2i+1)(2i\ 2i+2)$ for $1\leq i\leq n$, $\Psi (t_j)=1$ for $1\leq j\leq n+1$, and $\Psi (r)=(1\ 2)(3\ 4)\cdots (2n+1\ 2n+2)$. 
Hence $s_i$, $r_i$, and $t_j$ are parity-preserving and $r$ is parity-reversing, and these elements lie in $\LH $ by Lemma~\ref{lem_GW}. 
By Theorem~4.1 in~\cite{Omori-Yoshida}, we have the following proposition:

\begin{prop}\label{prop_gen_liftable-Hilden}
For $n\geq 1$, $\LH $ is generated by $s_i$, $r_i$ for $1\leq i\leq n$, $t_{j}$ for $1\leq j\leq n+1$, and $r$. 
\end{prop}

\section{Proof of main theorem}\label{section_lmod}

In this section, we prove Theorem~\ref{thm_lmod}. 
Assume that $n\geq 1$ and $k\geq 2$. 
Denote $s=s_n\cdots s_2s_1\in SW_{2n+2}$ and $s=\Gamma (s)\in \LHk $ (see Figure~\ref{fig_braid_s}). 
Theorem~\ref{thm_lmod} follows from the next proposition. 

\begin{prop}\label{prop_LM}
\begin{enumerate}
\item For $n\geq 1$, $\LHt =\Hil $ is generated by $s\sigma _1$, $s_{1}$, and $r_1$. 
\item For $n\geq 1$ and $k\geq 3$, $\LHk =\LH $ is generated by $sr$, $s_{1}$, and $r_1$.  
\end{enumerate}
\end{prop} 

\begin{figure}[h]
\includegraphics[scale=1.1]{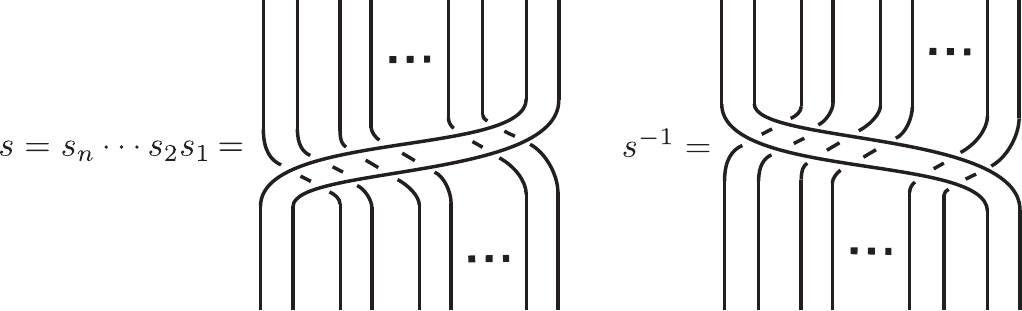}
\caption{The braid $s=s_n\cdots s_2s_1$ and $s^{-1}$.}\label{fig_braid_s}
\end{figure}

\begin{proof}[Proof of Proposition~\ref{prop_LM} for $k=2$]
Let $G$ be a subgroup of $\LHt =\Hil $ which is generated by $s\sigma _1$, $s_{1}$, and $r_1$. 
By Proposition~\ref{prop_gen_Hilden-grp}, $\Hil $ is generated by $s_i$, $r_i$ for $1\leq i\leq n$, and $\sigma _{2j-1}$ for $1\leq j\leq n+1$. 
Since $(s\sigma _1)^{-(i-1)}s_1(s\sigma _1)^{i-1}=s_i$ for $2\leq i\leq n$, we have $s_i\in G$ for $2\leq i\leq n$. 
Hence $s=s_n\cdots s_2s_1\in G$ and $\sigma _1=s^{-1}(s\sigma _1)\in G$. 
We can check that $r_i=s^{-(i-1)}r_1s^{i-1}\in G$ for $2\leq i\leq n$ and $\sigma _{2j-1}=s^{-(j-1)}\sigma _1s^{j-1}\in G$ for $2\leq j\leq n+1$. 
Therefore $\Hil =G$ and we have completed the proof of Proposition~\ref{prop_LM} for $k=2$.  
\end{proof}

\begin{proof}[Proof of Proposition~\ref{prop_LM} for $k\geq 3$]
Let $G$ be a subgroup of $\LHk =\LH $ which is generated by $sr$, $s_{1}$, and $r_1$. 
By Proposition~\ref{prop_gen_liftable-Hilden}, $\LH $ is generated by $s_i$, $r_i$ for $1\leq i\leq n$, $t_{j}$ for $1\leq j\leq n+1$, and $r$. 
Since $(sr)^{-(i-1)}s_1(sr)^{i-1}=s_i$ for $2\leq i\leq n$, we have $s_i\in G$ for $2\leq i\leq n$. 
Hence $s=s_n\cdots s_2s_1\in G$ and $r=s^{-1}(sr)\in G$. 
We can check that $r_i=s^{-(i-1)}r_1s^{i-1}\in G$ for $2\leq i\leq n$. 
Finally, by the relation~(4) of Theorem~4.1 in \cite{Omori-Yoshida}, we have the relation $r_1r_2\cdots r_ns_n\cdots s_2s_1t_{1}=1$. 
Thus, $t_{1}=s_1^{-1}s_2^{-1}\cdots s_n^{-1}r_n^{-1}\cdots r_2^{-1}r_1^{-1}\in G$ and $t_{j}=s^{-(j-1)}t_1s^{j-1}\in G$ for $2\leq j\leq n+1$.  
Therefore $\Hil =G$ and we have completed the proof of Proposition~\ref{prop_LM} for $k\geq 3$.  
\end{proof}


\par
{\bf Acknowledgement:} 
The author was supported by JSPS KAKENHI Grant Numbers JP19K23409 and JP21K13794.

\end{document}